\documentclass[reqno]{amsart}

\usepackage{amssymb, cancel, amsaddr}
\usepackage{amsmath}

\allowdisplaybreaks[4]

\newcommand{\IP}[2]{\left< #1 , #2 \right>}
\newcommand{\aIP}[2]{\left< #1 , #2 \right>}

\renewcommand{\H}{\vec{H}}

\newcommand{\R}{\ensuremath{\mathbb{R}}}

\newcommand{\SL}{\ensuremath{\mathcal{L{}}}}

\newtheorem{thm}{Theorem}{\bf}{\it}
{\bf}{\it}
{\bf}{\it}
\newtheorem{lem}[thm]{Lemma}{\bf}{\it}
\newtheorem*{conj}{Chen's Conjecture}{\bf}{\it}

\theoremstyle{definition}
\newtheorem*{rmk}{Remark}{\bf}{\rm}

\begin{document}

\title[Superbiharmonic submanifolds]{Chen's conjecture and $\varepsilon$-superbiharmonic
submanifolds of Riemannian manifolds}
\author{Glen Wheeler${}^{*}$}
\address{Otto-von-Guericke Universit\"at Magdeburg, Fakult\"at f\"ur Mathematik, Institut f\"ur
Analysis und Numerik, Universit\"atsplatz 2, 39106 Magdeburg, Germany}
\thanks{${}^*$: Financial support from the Alexander-von-Humboldt Stiftung is gratefully
acknowledged}

\begin{abstract}
B.-Y. Chen famously conjectured that every submanifold of Euclidean space with harmonic mean
curvature vector is minimal.
In this note we establish a much more general statement for a large class of submanifolds satisfying a
growth condition at infinity.
We discuss in particular two popular competing natural interpretations of the conjecture when the
Euclidean background space is replaced by an arbitrary Riemannian manifold.
Introducing the notion of $\varepsilon$-\emph{superbiharmonic submanifolds}, which contains each of the previous
notions as special cases, we 
prove that $\varepsilon$-superbiharmonic submanifolds of a complete Riemannian manifold which satisfy 
a growth condition at infinity are minimal.
\end{abstract}

\maketitle

\section{Introduction}

Suppose $M^m$ is a submanifold of a Riemannian manifold $(N^{m+n},\IP{\cdot}{\cdot})$ immersed via
$f:M^m\rightarrow N^{n+m}$ and equipped with the Riemannian metric induced via $f$.
Throughout we assume that all manifolds and mappings are proper and locally smooth.
Letting $\Delta$ denote the rough Laplacian, our goal is to determine sufficient conditions for the
validity of:

\begin{conj}[B.-Y. Chen \cite{Cconj}]
Suppose $f:M^m\rightarrow \R^{n+m}$ satisfies
\begin{equation}
\label{EQconj}
\Delta \H \equiv 0.
\end{equation}
Then $\H \equiv 0$.
\label{CJ}
\end{conj}

One fundamental difficulty is that the conjecture is a local question.  It appears at this time that
understanding the local structure of submanifolds satisfying \eqref{EQconj} to the point of
determining their minimality is a very difficult task.
Indeed, the conjecture continues to remain open with very little progress for hypersurfaces (where
the normal bundle is trivial) with intrinsic dimension greater than $4$ of Euclidean space.
Nevertheless, the study of the conjecture is quite active, with many partial results.
In \cite{chenishikawa98,Jiangcompactproof} Chen's conjecture is established for $m=2$ and $n=1$,
i.e., surfaces lying in $\R^3$.  Dimitri\'c \cite{dimitric92} proved the conjecture for $m=1$ and
$n$ arbitrary (curves in $\R^n$).  He also proved that if $f$ is additionally pseudo-umbilic then
the conjecture holds for $m\ne4$, and $n$ arbitrary, as well as that if $f$ possesses at most two
distinct principal curvatures and $m$ arbitrary, $n=1$ (hypersurfaces in $\R^{n+1}$). 
The case $m=3$, $n=1$ is proven in \cite{hasanisvlachos95}, and for surfaces lying in $\R^3$ it was
established by B.-Y. Chen in an unpublished work, as reported in \cite{Cconj}.

If $N^{n+m} = \R^{n+m}$ then submanifolds satisfying \eqref{EQconj} are also critical points of the
bi-energy $E(f) = \int_M |\tau(f)|^2d\mu$, where $\tau(f)$ is the tension field of $f$.
If the ambient space is curved however, then critical points of $E(f)$ satisfy instead
\begin{equation}
\label{EQbiharmmaps}
\Delta \H = R^N(e_i,\H)e_i\,,
\end{equation}
where $R^N$ is the curvature tensor of $N^{n+m}$ and $\{e_i\}$ is a local orthonormal frame of $M$.
Equation \eqref{EQbiharmmaps} is of course the condition for $f:M^m\rightarrow N^{n+m}$ to be an
intrinsic biharmonic map.
If $N$ has positive sectional curvature, then there are many well-known examples of non-minimal
solutions of \eqref{EQbiharmmaps}.
The so-called \emph{generalised Chen's conjecture} \cite{CMO01} states that if $N$ has everywhere
non-positive sectional curvature, then all solutions of \eqref{EQbiharmmaps} are minimal.
Although this turned out to be false \cite{OT10}, it remains interesting to determine sufficient
conditions which guarantee that solutions of \eqref{EQbiharmmaps} are minimal.
A survey of results in this direction can be found in \cite{survey}.

Now despite Chen's Conjecture being stated for submanifolds of Euclidean space, clearly equation
\eqref{EQconj} continues to make sense when the ambient space $N^{n+m}$ has some curvature.
It is thus not particularly clear which is the `correct' generalisation of the conjecture for
submanifolds of a Riemannian manifold $N^{n+m}$.
Given that there are many non-harmonic biharmonic maps (satisfying \eqref{EQbiharmmaps}), we find it
appropriate (as in \cite{biharmoniclifts2003} for example) to also investigate the minimality of
solutions of the original equation \eqref{EQconj} given by Chen, including considering the case
where the ambient space $N^{m+n}$ is curved.

One concept which generalises both \eqref{EQbiharmmaps} in the setting of a negatively curved
background space and the biharmonicity condition \eqref{EQconj} is
\begin{equation}
\label{EQsuperbiharmonic}
\IP{\Delta\H}{\H} \ge (\varepsilon-1)|\nabla H|^2\,,
\end{equation}
for $\varepsilon\in[0,1]$.
For $\varepsilon=0$ this implies $\Delta |\H|^2 \ge 0$, and so (somewhat respecting standard
convention, see Duffin \cite{D49} for example), we term solutions $f$ of \eqref{EQsuperbiharmonic}
\emph{$\varepsilon$-superbiharmonic submanifolds} or say that the mean curvature vector field is
\emph{$\varepsilon$-subharmonic}.

If $f$ is compact, then Chen's Conjecture is simple to prove.
This is a consequence of the argument used by Jiang \cite{Jiangcompactproof}.
An alternative method to obtain this result in flat space is as follows.
Compute
\[
0 = \int_M \IP{\H}{\Delta\H}d\mu
  = - \int_M |\nabla\H|^2d\mu\,
\]
so that $\nabla\H\equiv0$.
Now this implies
\[
0 = \int_M \Delta^2|f|^2d\mu
  = 2\int_M \Delta \Big( \IP{\H}{f} + m \Big)\,d\mu
  = 2\int_M |\H|^2d\mu
\]
and so we conclude $\H\equiv0$ and $f$ is minimal.
This simple argument is readily generalised to the setting of non-compact
solutions of \eqref{EQsuperbiharmonic}.
\footnote{Note added in proof: A theorem similar to Theorem \ref{TMgeneral} has been recently
obtained independently by Nakauchi, Urakawa and Gudmundsson \cite{NU}.}

\begin{thm}
Let $N^{n+m}$ be a complete Riemannian manifold.
Suppose $f:M^m\rightarrow N^{n+m}$ is $\varepsilon$-superbiharmonic in the sense that it satisfies
\eqref{EQsuperbiharmonic} for $\varepsilon\ge0$.
Assume in addition that $f$ satisfies the growth condition
\begin{equation}
\liminf_{\rho\rightarrow\infty}
  \frac{1}{\rho^2}\int_{f^{-1}(B_\rho)} |\H|^2d\mu
 = 0\,.
\label{EQgrowthcond}
\end{equation}
Then $\H \equiv 0$ and $f$ is minimal.
\label{TMgeneral}
\end{thm}

\begin{rmk}
Clearly $\R^{n+m}$ is complete and so the theorem applies in the setting of Chen's Conjecture.
This resolves the conjecture up to the growth condition \eqref{EQgrowthcond}.
\end{rmk}

\begin{rmk}
One must be careful in interpreting the growth condition in the case where $N$ is closed.
In this setting, the inverse image of the geodesic balls $f^{-1}(B_\rho)$ (geodesic in $N^{n+m}$)
will cover $M^m$ infinitely many times.  The growth condition \eqref{EQgrowthcond} thus becomes more
restrictive than it first appears (although not completely vacuous).
\end{rmk}


\section{Setting}

In this section we briefly set our notation and conventions.  We have as our main object of study an
immersion $f:M^m\rightarrow (N^{n+m},\IP{\cdot}{\cdot})$ and consider the $m$-dimensional Riemannian
submanifold $(M^m,g)$ with the metric $g$ induced by $f$, that is, given a local frame
$\tau_1,\ldots,\tau_m$ of the tangent bundle define the induced metric and associated induced volume
form by
\[
g_{ij} = \aIP{\partial_if}{\partial_jf}\, \qquad
d\mu = \sqrt{\text{det }g_{ij}}\, d\SL^{m}\,,
\]
where $d\SL^m$ denotes $m$-dimensional Lebesgue (or Hausdorff) measure.
Associated with $M^m$ is its (vector valued) second fundamental form, given by
\[
A_{ij} = \big(D_i\partial_{j}f\big)^\perp,
\]
where $D$ is the covariant derivative with respect to the Levi-Civita connection on $N$, and
$(\cdot)^\perp$ is the projection onto the normal bundle $(TM)^\perp$, which is given by
\[
X^\perp = X - X^\top = X - g^{ij}\IP{X}{\partial_if}\partial_jf 
\]
for a vector field $X:M^m \rightarrow N^{m+n}$.
The trace of $A$ under the metric $g$ is the mean curvature vector,
\[
\H = g^{ij}A_{ij}.
\]
The Levi-Civita connection $\nabla$ for $g$ is the unique metric connection on $M$ with coefficients
given in local coordinates by
\[
\nabla_{\tau_i}\tau_j = \Gamma_{ij}^k\tau_k,\quad\text{where}\quad \Gamma_{ij}^k
 = \frac{1}{2}g^{kl}(\partial_ig_{jl} + \partial_jg_{il} - \partial_lg_{ij}).
\]
Tracing $\nabla\nabla = \nabla_{(2)}$ over $g$ gives the metric or rough laplacian $\Delta$, and
$(p,q)$-tensor fields $T$ over $M^m$ are termed \emph{harmonic} if $\Delta T \equiv 0$ and
\emph{biharmonic} if
\[
\Delta^2 T
 := \Delta\Delta T
 \equiv 0
\]
on $M^m$.
In particular, the immersion $f:M^m \rightarrow N^{m+n}$ is itself called \emph{biharmonic} if
$\Delta^2f \equiv 0$ (cf. \eqref{EQconj}).
In this case we term $M^m$ (and $f(M^m)$) a \emph{biharmonic submanifold} of $N^{m+n}$.


\section{A local-global integral estimate}

Throughout the remainder of this note we shall use the abbreviations $M := M^m$ and $N := N^{m+n}$.
The ambient space $N$ will always be assumed to be complete.
We first prove the following lemma.

\begin{lem}
\label{LMhconst}
Suppose that $f:M\rightarrow N$ is an $\varepsilon$-superbiharmonic submanifold for some
$\varepsilon>0$ and
\begin{equation}
\liminf_{\rho\rightarrow\infty}
     \frac{1}{\rho^2}
     \int_{f^{-1}(B_\rho)} |\H|^2d\mu = 0\,.
\label{EQsmallgrowth}
\end{equation}
Then $\nabla\H\equiv0$.
\end{lem}
\begin{proof}
Suppose $\tilde{\eta}:N\rightarrow\R$ is a smooth cutoff function on a geodesic ball $B_\rho$
(centred anywhere) in $N$.
Clearly we can choose $\tilde{\eta}\in C_c^1(N)$ such that $\tilde{\eta}(q) = 1$ for $q\in B_\rho$,
$\tilde{\eta}(q) = 0$ for $q\not\in B_{2\rho}$, $\tilde{\eta}(q) \in [0,1]$ for all $q$, and
$|D\tilde{\eta}| \le \frac{c}{\rho}$ for some $c<\infty$.
Define $\eta(p) = (\tilde{\eta}\circ f)(p)$ for $p\in M$.
 
Let us use $\nabla^*$ to denote the divergence operator on $M$ with respect to $\nabla$.
Integrating $\nabla^*\Big(\IP{\H}{\nabla \H}\eta^2\Big)$ and using the divergence theorem we have
\[
\int_M \IP{\H}{\Delta \H}\eta^2\,d\mu
 + \int_M |\nabla \H|^2\eta^2\,d\mu
 + s\int_M \IP{\IP{\nabla \H}{\H} }{\nabla\eta}\eta\,d\mu
= 0\,,
\]
which implies
\begin{align*}
(\varepsilon-1)
\int_M |\nabla \H|^2\eta^2d\mu
  &\le
     \int_M \IP{\Delta \H}{\H}\eta^2\,d\mu
\\
  &=
    - \int_M |\nabla \H|^2\eta^2d\mu
    - 2\int_M \IP{\nabla \H}{\nabla\eta\,\H}\eta\,d\mu
\,,
\end{align*}
so
\begin{align*}
\int_{f^{-1}(B_\rho)} |\nabla \H|^2d\mu
 &\le
     \frac{c}{\varepsilon^2\rho^2}
     \int_{f^{-1}(B_{2\rho})} |\H|^2d\mu
\,.
\end{align*}
Recalling the assumption \eqref{EQsmallgrowth} and that $N$ is complete, the claim follows by taking
$\rho\rightarrow\infty$.
\end{proof}

It is important to note that the statement of the previous lemma is stronger than $\H$ being
parallel in the normal bundle; indeed, it is enough to guarantee that $f$ is minimal.

\begin{lem}
\label{LMh0}
Suppose the mean curvature $\H$ of $f:M\rightarrow N$ satisfies $\nabla\H \equiv 0$.  Then $f$ is
minimal.
\end{lem}
\begin{proof}
Let $p\in M$ and choose an orthonormal basis $\{\tau_i\}_{i=1}^m$ for $T_pM$.
We also choose an orthonormal basis $\{\nu^\alpha\}_{\alpha=1}^n$ of $(T_pM)^\perp$.
For any $i$, $j$, we have at $p$ that
\[
0 = \IP{\nabla_{\tau_i}\H}{\tau_j}
  = H_\alpha A^\alpha(\tau_i,\tau_j)\,.
\]
Tracing over the metric $g$, we conclude
\[
0 = \sum_{\alpha=1}^n(H_\alpha)^2
\]
and we are done.
\end{proof}

Combining Lemmas \ref{LMhconst} and \ref{LMh0} concludes the proof of Theorem \ref{TMgeneral}.

\section*{Acknowledgements}

This work was completed and announced during the 2011 SIAM PDE conference at Mission Valley.  The
author is grateful for the hospitality and atmosphere provided by SIAM at this event.
The author would like to thank in particular the participants of the mini-symposium ``Topics in
Higher Order Geometric Partial Differential Equations'' for lively and enlightening mathematical
discussions related to this work.

The author is supported by an Alexander von Humboldt research fellowship at the Otto-von-Guericke
Universit\"at Magdeburg.

\bibliographystyle{plain}
\bibliography{WheelerEpsSuperBihSubs}

\begin{thebibliography}{10}

\bibitem{CMO01}
R.~Caddeo, S.~Montaldo, and C.~Oniciuc.
\newblock Biharmonic submanifolds of {$\S^3$}.
\newblock {\em Internat. J. Math.}, 12(8):867--876, 2001.

\bibitem{Cconj}
B.Y. Chen.
\newblock Some open problems and conjectures on submanifolds of finite type.
\newblock {\em Soochow J. Math}, 17(2):169--188, 1991.

\bibitem{chenishikawa98}
B.Y. Chen and S.~Ishikawa.
\newblock Biharmonic pseudo-{R}iemannian submanifolds in pseudo-{E}uclidean
  spaces.
\newblock {\em Kyushu J. Math.}, 52(1):167--185, 1998.

\bibitem{dimitric92}
I.~Dimitric.
\newblock Submanifolds of {$\mathbb{E}^m$} with harmonic mean curvature vector.
\newblock {\em Bull. Inst. Math. Acad. Sinica}, 20(1):53--65, 1992.

\bibitem{D49}
R.J. Duffin.
\newblock On a question of {Hadamard} concerning super-biharmonic functions.
\newblock {\em J. Math. Phys}, 27:253--258, 1949.

\bibitem{hasanisvlachos95}
T.~Hasanis and T.~Vlachos.
\newblock Hypersurfaces in {$\mathbb{E}^4$} with harmonic mean curvature vector
  field.
\newblock {\em Math. Nachr.}, 172(1):145--169, 1995.

\bibitem{Jiangcompactproof}
G.Y. Jiang.
\newblock 2-harmonic maps and their first and second variational formulas.
\newblock {\em Chin. Ann. Math. Ser. A}, 7:389--402, 1986.

\bibitem{survey}
S.~Montaldo and C.~Oniciuc.
\newblock A short survey on biharmonic maps between {R}iemannian manifolds.
\newblock {\em Revista de la Uni{\'o}n Matem{\'a}tica Argentina}, 47(2):1--22,
  2006.

\bibitem{NU}
N.~Nakauchi, H.~Urakawa, and S.~Gudmundsson.
\newblock Biharmonic submanifolds in a {R}iemannian manifold with non-positive
  curvature.
\newblock {\em Arxiv preprint arXiv:1201.6457}, 2012.

\bibitem{OT10}
Y.L. Ou and L.~Tang.
\newblock The generalized chen's conjecture on biharmonic submanifolds is
  false.
\newblock {\em Arxiv preprint arXiv:1006.1838}, 2010.

\bibitem{biharmoniclifts2003}
M.A.J. Victoria and M.A.M. Bayo.
\newblock Biharmonic lifts by means of pseudo-{R}iemannian submersions in
  dimension three.
\newblock {\em Trans. Amer. Math. Soc.}, 355(1):169--176, 2003.

\end{thebibliography}

\end{document}